\documentclass[12pt]{article}
\usepackage{latexsym, amsmath, amsfonts, amsthm, amssymb}
\usepackage{times}
\usepackage{a4wide}

\def \RR {\mathbb R}

\def \EE {\mathbb E}

\def \eps {\varepsilon}

\def \cF {\mathcal F}

\newtheorem{theorem}{Theorem}[section]

\newtheorem{lemma}[theorem]{Lemma}
\newtheorem{proposition}[theorem]{Proposition}
\newtheorem{corollary}[theorem]{Corollary}
\newtheorem{remark}[theorem]{Remark}

\def\myffrac#1#2 in #3{\raise 2.6pt\hbox{$#3 #1$}\mkern-1.5mu\raise 0.8pt\hbox{$
#3/$}\mkern-1.1mu\lower 1.5pt\hbox{$#3 #2$}}

\def\qed{\hfill $\vcenter{\hrule height .3mm
\hbox {\vrule width .3mm height 2.1mm \kern 2mm \vrule width .3mm
height 2.1mm} \hrule height .3mm}$ \bigskip}

\begin{document}

\title{Eigenvalue distribution  of optimal transportation}
\author{Bo'az B. Klartag\thanks{School of Mathematical Sciences, Tel Aviv University, Tel Aviv 69978, Israel. E-mail: klartagb@tau.ac.il }, $\ $
Alexander V. Kolesnikov\thanks{Faculty of Mathematics, National Research University Higher School of Economics, Moscow, Russia. Email: sascha77@mail.ru}}
\date{}
\maketitle

\abstract{We investigate the Brenier map $\nabla \Phi$ between the uniform measures on
 two convex domains in $\RR^n$,
or more generally, between two log-concave probability measures on $\RR^n$.
We show that the eigenvalues of the Hessian matrix $D^2 \Phi$
exhibit remarkable concentration properties on a multiplicative scale, regardless of the choice
of the two measures or the dimension~$n$.
}

\section{Introduction}
\label{sec_intro}

Let $\mu$ and $\nu$ be two absolutely-continuous probability measures on $\RR^n$.
It was discovered by Brenier \cite{brenier} and McCann \cite{mccann}
that there exists a convex function $\Phi$ on $\RR^n$ with $(\nabla \Phi)_* \mu = \nu$, i.e.,
\begin{equation} \int_{\RR^n} b(\nabla \Phi(x)) d \mu(x) = \int_{\RR^n} b(x) d \nu(x) \label{eq_905} \end{equation}
for any $\nu$-integrable function $b: \RR^n \rightarrow \RR$. Moreover, the Brenier map $x \mapsto \nabla \Phi(x)$ is uniquely determined $\mu$-almost everywhere.
In this paper we consider the case where $\mu$ and $\nu$ are log-concave probability measures.
An absolutely-continuous probability measure on $\RR^n$ is called log-concave if it has a density $\rho$ which satisfies
$$ \rho \left( \lambda x + (1 - \lambda) y \right) \geq \rho(x)^{\lambda} \rho(y)^{1-\lambda} \qquad \qquad \qquad (x,y \in \RR^n, 0 < \lambda < 1). $$
The uniform measure on any convex domain is log-concave, as well as the Gaussian measure.
Write $Supp(\mu)$ for the interior of the support of $\mu$, which is an open, convex set in $\RR^n$. We make the assumption that
\begin{itemize}
\item[($\star$)] The function $\Phi$ is $C^2$-smooth in $Supp(\mu)$.
\end{itemize}
It follows from the works of Caffarelli \cite{Caf1, Caf2, ADM} that ($\star$) holds true
when each of the measures
$\mu$ and $\nu$ satisfies the following additional condition: Either the support of the measure
is the entire $\RR^n$, or else the support is a bounded, convex domain and
the density of the measure is bounded away from zero and from infinity
in this convex domain. It is fair to say that Caffarelli's regularity theory covers most
cases of interest, yet it is very plausible that ($\star$) is in fact always correct, without any
additional  conditions.

\medskip As it turns out, the positive-definite Hessian matrix $D^2 \Phi(x)$ exhibits remarkable regularity in the behavior of
its eigenvalues. We write $Var[X]$  for the variance
of the random variable $X$.
\begin{theorem} Let $\mu, \nu$ be absolutely-continuous, log-concave probability measures on $\RR^n$.
Let $\nabla \Phi$ be the Brenier map between $\mu$ and $\nu$, and assume ($\star$).
Write $0 < \lambda_1(x) \leq \ldots \leq \lambda_n(x)$ for the eigenvalues of the matrix $D^2 \Phi(x)$,
repeated according to their multiplicity.
Let $X$ be a random vector
in $\RR^n$ that is distributed according to  $\mu$. Then, for $i=1,\ldots,n$,
$$ Var \left[ \log \lambda_i(X) \right] \leq 4. $$
\label{thm1}
\end{theorem}

Thus, on a multiplicative scale, the eigenvalues of $D^2 \Phi$ are quite stable.
Note that the multiplicative scale is indeed the natural scale in the generality of Theorem \ref{thm1}:
By applying
appropriate linear transformations to $\mu$ and $\nu$, one may effectively multiply all eigenvalues by an arbitrary positive constant.
The variance bound in Theorem \ref{thm1} follows from a Poincar\'e inequality which we now formulate.
For $x \in Supp(\mu)$ set
$$ \Lambda(x) = \left( \log \lambda_1(x), \ldots, \log \lambda_n(x) \right).
$$
We write $| \cdot |$ for the standard Euclidean norm in $\RR^n$.

\begin{theorem} Under the notation and assumptions of Theorem \ref{thm1}, for any
locally-Lipschitz function $f: \RR^n \rightarrow \RR$ with $\EE \left| f(\Lambda(X)) \right| < \infty$,
$$ Var \left[ f(\Lambda(X)) \right] \leq 4 \EE |\nabla f|^2(\Lambda(X)),
$$
whenever the right-hand side is finite.
At the points in which $f$ is not continuously differentiable, we  define $|\nabla f|$
via (\ref{eq_232}) below.
\label{thm2}
\end{theorem}

Denote $\pi = \Lambda_*(\mu)$, the push-forward of the measure $\mu$ under the map $\Lambda$.
Theorem \ref{thm2} is a spectral gap estimate for the metric-measure space $(\RR^n, | \cdot |, \pi)$.
Gromov and Milman \cite{GM} proved that a spectral gap estimate implies exponential concentration of
Lipschitz functions. Therefore, Theorem \ref{thm2} admits the following immediate corollary:

\begin{corollary} We work under the notation and assumptions of Theorem \ref{thm1}.
Let $f: \RR^n \rightarrow \RR$ be a $1$-Lipschitz function (i.e., $|f(x) - f(y)| \leq |x-y|$).
Denote $A = \EE f(\Lambda(X))$. Then $A$ is finite and
$$ \EE \exp( c \left| f(\Lambda(X)) - A \right| ) \leq 2, $$
where $c > 0$ is a universal constant.
\label{cor1}
\end{corollary}

\begin{remark}{\rm
Corollary \ref{cor1} implies that $\EE e^{c |\Lambda(X)|} < \infty$.
Consequently, one may replace the condition $\EE \left| f(\Lambda(X)) \right| < \infty$ in Theorem \ref{thm2}
by the requirement that $e^{-c |x|} |f(x)|$ is bounded in $\RR^n$, for a certain universal constant $c > 0$. } \label{rem_1115} \end{remark}

Our next result is that the diagonal elements of the matrix $D^2 \Phi(x)$
are also concentrated on a logarithmic scale, pretty much like the eigenvalues.

\begin{theorem} We work under the notation and assumptions of Theorem \ref{thm1}.
Fix $v \in \RR^n$, let $H(x) = \log \left( D^2 \Phi(x) v \cdot v \right)$
and denote $Y = H(X)$. Then,
\begin{enumerate}
\item[(i)] $\displaystyle Var \left[ Y \right] \leq 4$.
\item[(ii)] For any locally-Lipschitz function $f: \RR \rightarrow \RR$ with $\EE \left| f(Y) \right| < \infty$,
$$
Var \left[ f(Y) \right] \leq 4 \EE \left| f^{\prime} \right|^2(Y).
$$
\item[(iii)] For any $1$-Lipschitz function $f: \RR \rightarrow \RR$, denoting $A = \EE f(Y)$ we have
that $A \in \RR$ and
$$ \EE \exp( c \left| f(Y) - A \right| ) \leq 2, $$
where $c > 0$ is a universal constant.
\end{enumerate}
\label{thm3}
\end{theorem}

All of the assertions made so far follow from Theorem \ref{thm_539} below, which is in fact a sound reformulation of \cite[Theorem 1.4]{K_moment}.
The results in \cite{K_moment} were obtained under a technical assumption dubbed ``regularity at infinity'',
which we shall address in this paper. Our argument is based on analysis of the
transportation metric: This means that we use the positive-definite Hessian $D^2 \Phi$
in order to define a Riemannian metric in $Supp(\mu)$. The weighted Riemannian manifold
$$ M_{\mu, \nu} = \left(Supp(\mu), D^2 \Phi , \mu \right) $$
was studied in \cite{kol}, where it was shown that the associated Ricci-Bakry-\'Emery tensor is non-negative
when $\mu$ and $\nu$ are log-concave. We will also  consider the map
$$ x \mapsto D^2 \Phi(x) $$
from $Supp(\mu) \subseteq \RR^n$ into the space of positive-definite matrices. The space of positive-definite matrices is endowed with a natural Riemannian metric,
which fits very nicely with computations related to the weighted Riemannian manifold $M_{\mu, \nu}$. This leads
to a certain Poincar\'e inequality with respect to the standard Riemannian metric on the space of
positive-definite matrices, formulated in Theorem \ref{thm_539} below .

\medskip We have tried to make the exposition self-contained, apart from the regularity theory of mass-transport. The rest of this paper is organized as follows: In Section \ref{sec_pos} we recall some well-known
constructions related to positive-definite matrices. In Section \ref{sec_gamma_2} and Section \ref{sec_dual} we prove the main results
under regularity assumptions by employing the Bakry-\'Emery $\Gamma_2$-calculus. Section \ref{sec_regularity} is devoted to the elimination of these regularity assumptions.
In Section \ref{sec6} we complete the proofs of the theorems formulated above.
We write $x \cdot y$ for the standard scalar product of $x, y \in \RR^n$.
We denote derivatives by
$\partial_k f = f_k = \partial f / \partial x_k$
and
$f_{ij} = \partial^2 f / (\partial x_i \partial x_j)$.
By a smooth function we mean a $C^{\infty}$-smooth one.
We write $\log$ for the natural logarithm, and $Tr(A)$ stands for the trace of the matrix $A$.

\medskip
\emph{Acknowledgements.} We would like to thank Emanuel Milman for interesting discussions.
The first named author was supported by a grant from the European Research Council (ERC).
The second named author was supported by RFBR project 12-01-33009 and  the DFG project  CRC 701.
This study (research grant No 14-01-0056) was supported by The National Research University–-Higher School of Economics' Academic Fund Program in 2014/2015.

\section{Positive-definite quadratic forms}
\label{sec_pos}

This section surveys  standard material on positive-definite matrices.
Denote by $M_n^+(\RR)$ the collection of all symmetric, positive-definite $n \times n$ matrices.
For a function $f: (0, \infty) \rightarrow \RR$ and $A \in M_n^+(\RR)$ we may define the symmetric matrix $f(A)$
via the spectral theorem. In other words,
$$ f \left( \sum_{i=1}^n \lambda_i v_i \otimes v_i \right) = \sum_{i=1}^n f(\lambda_i) v_i \otimes v_i $$
for any orthonormal basis $v_1,\ldots,v_n \in \RR^n$ and $\lambda_1,\ldots,\lambda_n > 0$, where we write $x \otimes x = (x_i x_j)_{i,j=1,\ldots,n}$
for $x = (x_1,\ldots,x_n) \in \RR^n$.

\begin{lemma} For any $A, B \in M_n^+(\RR)$,
\begin{equation} \left \| \log \left( A^{1/2} B A^{1/2} \right) \right \|_{HS} \leq \left \| \log(A) \right \|_{HS} + \left \| \log(B) \right \|_{HS}
\label{eq_1119} \end{equation}
where $\| \cdot \|_{HS}$ stands for the Hilbert-Schmidt norm.
\label{lem_1553}
\end{lemma}

\begin{proof}
For an  $n \times n$  matrix $T$ and $k = 1,\ldots,n$ we define
\begin{equation}  D_k(T) = \sup_{E \subseteq \RR^n \atop{\dim(E) = k}} \frac{Vol_k(T(B^n \cap E))}{Vol_k(B^n \cap E)},
\label{eq_1045} \end{equation}
where $B^n = \{ x \in \RR^n \, ; \, |x| < 1 \}$, and the supremum in (\ref{eq_1045}) runs over all $k$-dimensional
subspaces in $\RR^n$. Thus, an application of the linear transformation $A$ may increase $k$-dimensional volumes
by a factor of at most $D_k(A)$. It follows that
for any  $n \times n$  matrices $A$ and $B$,
\begin{equation}  D_k(AB) \leq D_k(A) D_k(B) \qquad \qquad (k=1,\ldots,n).
\label{eq_1533} \end{equation}
In the case where $A \in M_n^+(\RR)$, we have $D_k(A) = \prod_{i=1}^k \lambda_i$,
where $\lambda_1 \geq \lambda_2 \geq \ldots \geq \lambda_n > 0$ are the eigenvalues of $A$.
Assume that $A, B \in M_n^+(\RR)$. Denote the eigenvalues of the symmetric, positive-definite
matrix $A^{1/2} B A^{1/2}$ by $e^{\gamma_1} \geq \ldots \geq e^{\gamma_n} > 0$. Then, for $k=1,\ldots,n$,
\begin{equation}
\prod_{i=1}^k e^{\gamma_i} = D_k \left(  A^{1/2} B A^{1/2} \right) \leq D_k(A^{1/2}) D_k(B) D_k(A^{1/2})
= D_k(A) D_k(B) = \prod_{i=1}^k (e^{\alpha_i} e^{\beta_i}),
\label{eq_1056}
\end{equation}
where $e^{\alpha_1} \geq \ldots \geq e^{\alpha_n} > 0$ are the eigenvalues of $A$, and
$e^{\beta_1} \geq \ldots \geq e^{\beta_n} > 0$ are the eigenvalues of $B$. We will next apply a lemma of Weyl \cite{weyl}, see also Polya \cite{polya}.
According to the inequality of Weyl and Polya, the inequalities (\ref{eq_1056}) entail that
\begin{equation} \sum_{i=1}^n h(\gamma_i) \leq \sum_{i=1}^n h(\alpha_i + \beta_i) \label{eq_1104}
\end{equation}
for any convex, non-decreasing function $h: \RR \rightarrow \RR$. For $t \in \RR$ denote $t_+ = \max \{ t, 0 \}$.
The function $t \mapsto (t_+)^2$ is convex and non-decreasing, hence from (\ref{eq_1104}),
\begin{equation}
\sum_{i=1}^n ((\gamma_i)_+)^2 \leq \sum_{i=1}^n ((\alpha_i + \beta_i)_+)^2. \label{eq_1532}
\end{equation}
By using (\ref{eq_1533}) for the inverse matrices, we conclude that for $k=1,\ldots,n$,
$$
\prod_{i=n-k+1}^n e^{-\gamma_{i}} = D_k \left(  A^{-1/2} B^{-1} A^{-1/2} \right) \leq D_k(A^{-1}) D_k(B^{-1})
= \prod_{i=n-k+1}^n (e^{-\alpha_i} e^{-\beta_i}).
$$
The  inequality of Weyl and Polya now implies that $\sum_{i=1}^n h(-\gamma_i) \leq \sum_{i=1}^n h(-\alpha_i - \beta_i)$
for any convex, non-decreasing function $h$. By again using  $h(t) = (t_+)^2$, we get
\begin{equation}
\sum_{i=1}^n ((-\gamma_i)_+)^2 \leq \sum_{i=1}^n ((-\alpha_i - \beta_i)_+)^2. \label{eq_1538}
\end{equation}
Adding (\ref{eq_1532}) and (\ref{eq_1538}), we finally obtain
\begin{equation}
\sum_{i=1}^n \gamma_i^2 \leq \sum_{i=1}^n (\alpha_i + \beta_i)^2 \leq \left( \sqrt{\sum_{i=1}^n \alpha_i^2} +
\sqrt{\sum_{i=1}^n \beta_i^2} \right)^2,
\label{eq_1539}
\end{equation}
where we used the Cauchy-Schwartz inequality in the last passage.
By taking the square root of (\ref{eq_1539}) we deduce (\ref{eq_1119}).
\end{proof}

For  two matrices $A, B \in M_n^+(\RR)$  set
\begin{equation}
 dist(A, B) = \left \| \log \left( A^{-1/2} B A^{-1/2} \right) \right \|_{HS}. \label{eq_1924}\end{equation}
Equivalently,  $dist(A, B)$ equals $\sqrt{\sum_i \log^2 \lambda_i}$,
where $\lambda_1,\ldots,\lambda_n > 0$ are the eigenvalues of the matrix $A^{-1} B$
which is conjugate to $A^{-1/2} B A^{-1/2}$. The latter equivalent definition of $dist$
shows that for any invertible $n \times n$ matrix $T$,
\begin{equation} dist \left( A, B \right) = dist \left( T^t A T, T^t B T \right) \qquad \qquad (A,B \in M_n^+(\RR^n)), \label{eq_1556}
\end{equation}
where $A^t$ is the transpose of the matrix $A$.
Observe too that $dist \left( A, B \right) = dist \left( A^{-1}, B^{-1} \right)$ for any $A, B \in M_n^+(\RR)$.
Lemma \ref{lem_1553} states that for  $A, B \in M_n^+(\RR)$,
\begin{equation}
dist(A,B) \leq dist(A, {\rm Id}) + dist({\rm Id}, B),
\label{eq_1554}
\end{equation}
where ${\rm Id}$ is the identity matrix.
From (\ref{eq_1556}) and (\ref{eq_1554}) one realizes that $dist$ satisfies the triangle inequality
in $M_n^+(\RR)$, hence it is a metric.
For $A \in M_n^+(\RR^n)$ and a symmetric $n \times n$ matrix $B$ we denote
$$ \| B \|_{A} = \left \| A^{-1/2} B A^{-1/2} \right \|_{HS} = \sqrt{ Tr \left[ (A^{-1} B)^2 \right] }.
$$
For a smooth curve $\gamma: [a,b] \rightarrow M_n^+(\RR)$ set
\begin{equation}
 Length(\gamma) = \int_a^b \left \| \dot{\gamma}(s) \right \|_{\gamma(s)} ds, \label{eq_519} \end{equation}
where $\dot{\gamma}(s) = \frac{d \gamma(s)}{ds}$ is a symmetric $n \times n$ matrix.
Then $Length$ is invariant under conjugations. That is, the length of the curve $\gamma(s)$
equals that of the curve $T^t \gamma(s) T$ for any invertible $n \times n$ matrix $T$.

\begin{lemma}
\begin{enumerate}
\item[(i)] For any $A \in M_n^+(\RR^n)$ and a symmetric $n \times n$ matrix $B$,
\begin{equation}  \lim_{\eps \rightarrow 0} \frac{dist^2(A + \eps B, A)}{\eps^2} = \| B \|^2_{A} =  Tr \left[ (A^{-1} B)^2 \right].
\label{eq_1931} \end{equation}
\item[(ii)] Let $A, B \in M_n^+(\RR^n)$ and consider the curve
$$ \gamma_{A, B}(s) = A^{1/2} \left( A^{-1/2} B A^{-1/2} \right)^s A^{1/2} \qquad \qquad (0 \leq s \leq 1). $$
Then $\gamma_{A,B}$ is a curve connecting $A$ and $B$
with $Length(\gamma_{A,B}) = dist(A,B)$.
\end{enumerate}
\label{lem_454}
\end{lemma}

\begin{proof} The invariance property (\ref{eq_1556}) implies that
$$ dist (A + \eps B, A) = dist ({\rm Id} + \eps A^{-1/2} B A^{-1/2}, {\rm Id}). $$
It therefore suffices to prove (i) under the additional assumption that $A = {\rm Id}$.
Let $\lambda_1,\ldots,\lambda_n > 0$ be the eigenvalues of $B$.
It follows  from (\ref{eq_1924}) that
$$  \lim_{\eps \rightarrow 0} \frac{dist^2 ({\rm Id} + \eps B, {\rm Id})}{\eps^2} =
\lim_{\eps \rightarrow 0} \frac{\sum_{i=1}^n \log^2 (1+ \eps \lambda_i) }{\eps} = \sum_{i=1}^n \lambda_i^2,
$$
and (i) follows from the fact that $\| B \|_A^2  = \sum_i \lambda_i^2$.
We now turn to the proof of (ii). Again, we may reduce matters to the case where $A = {\rm Id}$ by noting that
$$ \gamma_{A,B}(s) = A^{1/2} \gamma_{{\rm Id}, A^{-1/2} B A^{-1/2}}(s) A^{1/2} \qquad \qquad (0 \leq s \leq 1). $$
Abbreviate $\gamma(s) = \gamma_{A,B}(s) = \gamma_{{\rm Id}, B}(s)$. Since $\gamma(s) = B^s$ then $\dot{\gamma}(s) = B^s \log(B)$ and hence, for any $0 \leq s \leq 1$,
$$ \| \dot{\gamma}(s) \|_{\gamma(s)} = \left \| B^{-s/2} \left( B^s \log(B) \right) B^{-s/2} \right \|_{HS}  = \left \| \log(B) \right \|_{HS}  = dist ({\rm Id}, B). $$
From the definition (\ref{eq_519}) it follows that $Length (\gamma) = dist ({\rm Id},B)$, and (ii) is proven.
\end{proof}

The right-hand side of (\ref{eq_1931}) depends quadratically on $B$,
and therefore Lemma \ref{lem_454} tells us that our distance function $dist$ on $M_n^+(\RR)$ is induced by a Riemannian metric. We refer to this
Riemannian metric as the {\it standard Riemannian metric} on $M_n^+(\RR)$.
The next two lemmas describe certain Lipschitz functions on $M_n^+(\RR)$.

\begin{lemma} Fix $v \in \RR^n$ and set $f(A) = \log (A v \cdot v)$ for $A \in M_n^+(\RR)$.
Then $f$ is a $1$-Lipschitz function with respect to the standard Riemannian metric on $M_n^+(\RR)$.
\label{lem_2201}
\end{lemma}

\begin{proof} The map $f$ is clearly smooth.
Fix $A \in M_n^+(\RR)$ and let us show that the norm of the Riemannian gradient of $f$ at the point $A$
is bounded by one.
For any symmetric $n \times n$ matrix $B$ we have
$$ \left. \frac{d}{dt} f(A + t B) \right|_{t=0} = \frac{B v \cdot v}{Av \cdot v}. $$
Thus, in order to prove the lemma, it suffices to show that
\begin{equation}
\frac{B v \cdot v}{Av \cdot v} \leq \| B \|_A = \| A^{-1/2} B A^{-1/2} \|_{HS}.
\label{eq_615}
\end{equation}
By switching to another orthonormal basis, if necessary, we may assume that $A$ is a diagonal matrix. Denote by $\lambda_1,\ldots,\lambda_n > 0$ the numbers on the diagonal of $A$.
Denote $B = (b_{ij})_{i,j=1,\ldots,n}$ and $v = (v_1,\ldots,v_n) \in \RR^n$. From the Cauchy-Schwartz inequality,
$$ \sum_{i,j=1}^n b_{ij} v_i v_j \leq \sqrt{ \sum_{i,j=1}^n \frac{b_{ij}^2}{\lambda_i \lambda_j} } \sqrt{ \sum_{i,j=1}^n \lambda_i \lambda_j v_i^2 v_j^2 }
= \sqrt{ \sum_{i,j=1}^n \frac{b_{ij}^2}{\lambda_i \lambda_j} } \left( \sum_{i=1}^n \lambda_i v_i^2 \right),
$$
which is equivalent to the desired inequality (\ref{eq_615}).
\end{proof}

\begin{lemma} For $A \in M_n^+(\RR)$ denote its eigenvalues by
$ \lambda_1(A) \geq \ldots \geq \lambda_n(A) > 0 $.
Consider the map $\Lambda: M_n^+(\RR) \rightarrow \RR^n$ defined via
\begin{equation}  \Lambda(A) = \left( \log(\lambda_1(A)), \ldots, \log(\lambda_n(A)) \right). \label{eq_1007} \end{equation}
Then $\Lambda$ is a $1$-Lipschitz map, with respect to the standard Riemannian metric on $M_n^+(\RR)$, and the standard
Euclidean metric on $\RR^n$. \label{lem_2200}
\end{lemma}

\begin{proof} Let $\cF \subseteq M_n^+(\RR)$ be the collection of all positive-definite, symmetric matrices
with $n$ distinct eigenvalues. Then $\cF$ is an open, dense set.
The function $\Lambda$ is continuous,
since the eigenvalues vary continuously with the matrix. It therefore suffices
to prove that $$ |\Lambda(A_1) - \Lambda(A_2)| \leq dist(A_1, A_2) \qquad \qquad \text{for} \ A_1, A_2 \in \cF. $$
Fix $A_1, A_2 \in \cF$. Consider the curve $\gamma(s) = \gamma_{A_1,A_2}(s / dist(A_1,A_2))$
where $\gamma_{A_1,A_2}(s)$
is as in Lemma \ref{lem_454}. Then $\gamma$ is a length-minimizing curve between $A_1$ and $A_2$, parametrized
by Riemannian arclength. We claim that $\gamma(s) \in \cF$ for all but finitely many values
of $s$. Indeed, the resultant of $\gamma(s)$ is a real-analytic function of $s$ which is not identically zero, hence
its zeros are isolated. Since $\Lambda \circ \gamma$ is continuous, in order to prove the lemma
it suffices to show that
\begin{equation}
 \left| \frac{d \Lambda(\gamma(s))}{ds} \right| \leq 1 \label{eq_2145} \end{equation}
for all $s$ with $\gamma(s) \in \cF$. Let us fix $s_0$ with $\gamma(s_0) \in \cF$.
Denote $A = \gamma(s_0)$ and $B = \dot{\gamma}(s_0)$. Since $\gamma$ is parameterized  by arclength, then
\begin{equation} \| B \|_A = \| A^{-1/2} B A^{-1/2} \|_{HS} = 1. \label{eq_2153} \end{equation}
Let $v_1,\ldots,v_n \in \RR^n$ be the orthonormal basis of eigenvectors that corresponds
to the eigenvalues $\lambda_1(A),\ldots,\lambda_n(A)$ of the matrix $A$.
Then,
\begin{equation}  \left. \frac{d \lambda_i(\gamma(s))}{ds} \right|_{s=s_0} = B v_i \cdot v_i \qquad \qquad (i=1,\ldots,n). \label{eq_2127}
\end{equation}
The relation (\ref{eq_2127}) is standard, see, e.g. Reed and Simon \cite[Section XII.1]{RS}. Consequently,
\begin{equation} \left. \frac{d \Lambda(\gamma(s))}{ds} \right|_{s=s_0} = \left( \frac{B v_1 \cdot v_1}{\lambda_1(A)}, \ldots,
 \frac{B v_n \cdot v_n}{\lambda_n(A)} \right). \label{eq_2155} \end{equation}
However, by (\ref{eq_2153}),
\begin{equation} \sum_{i=1}^n \left( \frac{B v_i \cdot v_i}{\lambda_i(A)} \right)^2 = \sum_{i=1}^n \left( A^{-1/2} B A^{-1/2} v_i \cdot v_i \right)^2
\leq \| A^{-1/2} B A^{-1/2} \|_{HS}^2 = 1. \label{eq_2154_} \end{equation}
Now (\ref{eq_2145}) follows from (\ref{eq_2155}) and (\ref{eq_2154_}).
\end{proof}

\begin{corollary} Whenever $A$ and $B$ are positive-definite
$n \times n$ matrices,
$$ \sum_{i=1}^n \log^2 \frac{\lambda_i}{\mu_i} \leq \left \| \log \left(A^{-1/2} B A^{-1/2} \right) \right\|_{HS}^2
$$
where $\lambda_1 \geq \ldots \geq \lambda_n > 0$ are the eigenvalues of $A$, and
$\mu_1 \geq \ldots \geq \mu_n > 0$ are the eigenvalues of $B$.
\end{corollary}

\section{Bakry-\'Emery $\Gamma_2$-calculus}
\label{sec_gamma_2}

 Let $\mu$ and $\nu$ be two absolutely-continuous, log-concave probability measures on $\RR^n$.
Assume that $d \mu = e^{-V(x)} dx$ and $d \nu = e^{-W(x)} dx$,
for certain smooth, convex functions $V, W: \RR^n \rightarrow \RR$. Let $\nabla \Phi$ be the Brenier map between $\mu$ and $\nu$.
Caffarelli's regularity theory states that $\Phi: \RR^n
\rightarrow \RR$ is a smooth, convex function. Therefore (\ref{eq_905}) implies that
the transport equation
\begin{equation} -V(x) = \log \det D^2 \Phi(x) - W(\nabla \Phi(x)) \label{eq_2154} \end{equation}
holds everywhere in $\RR^n$. In particular, the matrix $D^2 \Phi(x) = \left( \Phi_{ij}(x) \right)_{i,j=1,\ldots,n}$ is invertible and hence positive-definite for any $x \in \RR^n$.
The inverse matrix to $D^2 \Phi(x)$ is denoted by
$\left( D^2 \Phi(x) \right)^{-1} = \left( \Phi^{ij}(x) \right)_{i,j=1,\ldots,n}$.
We use the Einstein summation convention, thus an index that appears twice in an expression,
once as a subscript and once as a superscript, is being summed upon.
We also use abbreviations such as $\Phi^{i}_{jk} = \Phi^{i \ell} \Phi_{jk \ell}$
and $\Phi^{ij}_{k} = \Phi^{i \ell} \Phi^{j m} \Phi_{km \ell}$.
 Differentiating (\ref{eq_2154}), we obtain
\begin{equation}
 V_j(x) = - \Phi_{ji}^i(x) + \sum_{i=1}^n \Phi_{ij}(x) W_i(\nabla \Phi(x))  \qquad \qquad (j=1,\ldots,n, \ x \in \RR^n).
\label{eq_2207}
\end{equation}
Following \cite{kol}, we use the positive-definite matrices $D^2 \Phi(x)$ in order to induce a Riemannian metric on $\RR^n$,
and consider the weighted Riemannian manifold
$$ M = M_{\mu, \nu} = \left( \RR^n, D^2 \Phi, \mu \right). $$
See Grigor'yan \cite{gri} and Bakry, Gentil and Ledoux \cite{BGL} for background on weighted Riemannian manifolds
and the $\Gamma_2$-calculus. For a smooth function $u: \RR^n \rightarrow \RR$ we have $|\nabla_M u|_M^2 = \Phi^{ij} u_i u_j$
where $|\nabla_M u|_M^2$ stands for the square of the Riemannian norm of the Riemannian gradient of $u$.
The Dirichlet form associated with the weighted Riemannian manifold $M_{\mu, \nu}$ is defined, for smooth functions $u, v: \RR^n \rightarrow \RR$, via
$$ \Gamma(u, v) = \int_{\RR^n} \langle \nabla_M u, \nabla_M v \rangle_M \, d \mu = \int_{\RR^n} \left( \Phi^{ij} u_i v_j \right) d \mu $$
whenever the integral converges. The Laplacian associated
with the weighted Riemannian manifold $M_{\mu, \nu}$ is defined, for a smooth function $u: \RR^n \rightarrow \RR$, by
\begin{equation}  L u = \Phi^{ij} u_{ij} - \sum_{j=1}^n W_j(\nabla \Phi(x)) u_j
= \Phi^{ij} u_{ij} - \left( \Phi_i^{ij} + \Phi^{ij} V_i \right)  u_j, \label{eq_434} \end{equation}
where the last equality holds in view of (\ref{eq_2207}). Integrating by parts, we verify that
$$ -\int_{\RR^n} (Lu) v d \mu =  -\int_{\RR^n} \left(\Phi^{ij} u_{ij} - \left[ \Phi_i^{ij} + \Phi^{ij} V_i \right]  u_j \right) v e^{-V}  =
\int_{\RR^n} \left( \Phi^{ij} u_i v_j \right) d \mu = \Gamma(u,v) $$
for any smooth functions $u,v : \RR^n \rightarrow \RR$, one of whom is compactly-supported. The next step
is to consider the {\it Carr\'e du Champ} of $M_{\mu, \nu}$:  As in Bakry and \'Emery \cite{BE}, for a smooth function $u: K \rightarrow \RR$ we define
\begin{equation} \Gamma_2(u) = \frac{1}{2} L \left( |\nabla_M u|^2_M \right) - \langle \nabla_M u, \nabla_M (L u) \rangle_M = \frac{1}{2} L \left( \Phi^{ij} u_i u_j \right) - \Phi^{ij} (L u)_i u_j. \label{eq_1400} \end{equation}

\begin{lemma}
For any smooth function $u: \RR^n \rightarrow \RR$ we have the pointwise inequality
$$\Gamma_2(u) \geq \frac{1}{4} \Phi^{i k}_{\ell} \Phi^{j \ell}_k u_i u_{j}.
$$
\label{lem_1035}
\end{lemma}

Lemma \ref{lem_1035} is proven in \cite{K_moment} by introducing a K\"ahler structure and interpreting  the left-hand side of (\ref{eq_1042}) below
as the Hilbert-Schmidt norm of a certain Hessian operator restricted to a subspace. There are several additional ways
to prove Lemma \ref{lem_1035}. The brute-force way involves a tedious but straightforward computation which shows that
$$
 \Gamma_2(u) = \Phi^{kl} \Phi^{ij} u_{i k } u_{j \ell} - \Phi^{i j k} u_{ij} u_k +
 \frac{1}{2} \left( \Phi^{i k}_{\ell} \Phi^{j \ell}_k + \Phi^{i k} \Phi^{j \ell} V_{k \ell}   \right) u_i u_j
 + \frac{1}{2} \sum_{i,j=1}^n  (W_{ij} \circ \nabla \Phi) u_i u_j.
$$
This computation is more or less equivalent to reproving Bochner's formula.
Then, one proves the pointwise inequality
\begin{equation}
\Phi^{kl} \Phi^{ij} u_{i k } u_{j \ell} - \Phi^{i j k} u_{ij} u_k +  \frac{1}{4} \Phi^{i k}_{\ell} \Phi_{k}^{j \ell} u_i u_{j} \geq 0,
\label{eq_1042} \end{equation}
by representing the left-hand side of (\ref{eq_1042}) as the trace of the square of the matrix $B = (b_i^j)_{i,j=1,\ldots,n}$
where $b_i^j = \Phi^{j k} u_{k i} - \frac{1}{2} \Phi_i^{jk} u_k$. The product $A = (D^2 \Phi) B$ is a symmetric matrix,  hence
 $$ Tr \left( B^2 \right) =
 Tr \left[ \left( (D^2 \Phi)^{-1/2} A (D^2 \Phi)^{-1/2} \right)^2 \right] \geq 0. $$
Lemma \ref{lem_1035}
follows from (\ref{eq_1042}) and from the fact that $D^2 V$ and $D^2 W$ are positive semi-definite matrices.

\medskip  Another  approach to Lemma \ref{lem_1035} is to  use the notation of Riemannian geometry
as in \cite{kol}, and use
the Bochner formula. We first observe that   identity
(\ref{eq_2207}) in the case $j=1$ has the simple form
\begin{equation} \label{eq_440} L \Phi_1 = -V_1.
\end{equation}
Differentiating (\ref{eq_440}) and using $\partial_k (\Phi^{ij}) = -\Phi^{ij}_k$, we obtain
\begin{equation}
L(\Phi_{11}) - \Phi^{j k}_1 \Phi_{1jk} - \sum_{j, k=1}^n    \Phi_{j 1} \Phi_{1 k} \left( W_{j k} \circ \nabla \Phi \right) = -V_{11}.
\label{eq_445}
\end{equation}
The  Bochner-Lichnerowicz-Weitzenb\"ock formula
 states that for any smooth $u: \RR^n \rightarrow \RR$,
\begin{equation}
\Gamma_2(u)  = \| D_M^2 u \|_M^2  + Ric_M (\nabla_M u, \nabla_M u),
\label{eq_1051}
\end{equation}
where $ \| D^2_M u \|_M^2$ is the Hilbert-Schmidt norm of the Riemannian Hessian of $u$, and $Ric_M$ is the Bakry-\'Emery-Ricci tensor
of the weighted Riemmannian manifold $M = M_{\mu, \nu}$. Let us analyze the term in (\ref{eq_1051}) involving
the Hessian of $u$. The Christofell symbols of our Riemannian metric are $ \Gamma_{ij}^k = \frac{1}{2} \Phi_{ij}^k$, and therefore
$(D_M^2 u)_{ij} = u_{ij} - \frac{1}{2} \Phi_{ij}^k u_k$ and
$$ \| D_M^2 u \|_M^2 = \Phi^{ik} \Phi^{jm} \left( u_{ij} - \frac{1}{2} \Phi_{ij}^\ell u_\ell \right)
 \left( u_{mk} - \frac{1}{2} \Phi_{mk}^s u_s \right).
$$
In the particular case where $u = \Phi_1$, we obtain $( D_M^2 \Phi_1)_{jk} = \frac{1}{2} \Phi_{1jk}$ and hence
$\| D_M^2 \Phi_1 \|_M^2 = \frac{1}{4} \Phi_{1j}^k \Phi_{1k}^j$. Furthermore,
the vector field $\nabla_M \Phi_1$ satisfies $\nabla_M \Phi_1 = \partial / \partial x_1$ and $|\nabla_M \Phi_1|^2_M = \Phi_{11}$.
Since $L \Phi_1 = -V_1$, the Bochner formula (\ref{eq_1051}) for $u = \Phi_1$ takes the form
\begin{align} \nonumber
\frac{1}{2} L \left( \Phi_{11} \right) & = -\langle \nabla_M \Phi_1, \nabla_M V_1 \rangle_M + \frac{1}{4} \Phi_{1j}^k \Phi_{1k}^j  + Ric_M (\nabla_M u, \nabla_M u)
\\ & = -V_{11} + \frac{1}{4} \Phi_{1j}^k \Phi_{1k}^j + (Ric_M)_{11}.
\label{eq_1051_}
\end{align}
From (\ref{eq_445}) and (\ref{eq_1051_}) we obtain a formula for the Bakry-\'Emery-Ricci tensor:
$$(Ric_M)_{11} = \frac{1}{4} \Phi_{1j}^k \Phi_{1k}^j + \frac{1}{2} V_{11} + \frac{1}{2} \sum_{j, k=1}^n    \Phi_{j 1} \Phi_{1 k} \left( W_{j k} \circ \nabla \Phi \right). $$
It is clear that there is nothing special about the derivative $u = \Phi_1$, and that we could have repeated the argument
with $u = \nabla \Phi \cdot \theta$ for any $\theta \in \RR^n$. We thus obtain the formula
\begin{equation}
(Ric_M)_{i \ell} = \frac{1}{4} \Phi_{i j}^k \Phi_{\ell k}^j + \frac{1}{2} V_{i \ell} + \frac{1}{2} \sum_{j, k=1}^n    \Phi_{j i} \Phi_{\ell k} \left( W_{j k} \circ \nabla \Phi \right).
\label{eq_458} \end{equation}
Since $D^2 V$ and $D^2 W$ are positive semi-definite, then for any smooth $u: \RR^n \rightarrow \RR$,
$$ \Gamma_2(u) \geq Ric_M (\nabla_M u, \nabla_M u) \geq \frac{1}{4} \Phi_{j}^{ik} \Phi_{k}^{j \ell} u_i u_{\ell} $$
and the third proof of Lemma \ref{lem_1035} is complete.

\medskip Having finished with Lemma \ref{lem_1035}, let us introduce one of the main ideas in
this paper, which was absent from \cite{K_moment}. The idea is to consider the map
\begin{equation}
 \RR^n \ni x \mapsto D^2 \Phi(x) \in M_n^+(\RR).
 \label{eq_1318}
 \end{equation}
Denote by $(g_{ij}(x))_{i,j=1,\ldots,n}$ the pull-back of the standard
Riemannian metric on $M_n^+(\RR)$ via  the map (\ref{eq_1318}).
It follows from  Lemma \ref{lem_454} that $g_{ij}$ is given by the formula
\begin{equation}
 g_{ij} = Tr \left[ (D^2 \Phi)^{-1} \cdot \partial_i \left( D^2 \Phi \right)
\cdot (D^2 \Phi)^{-1} \cdot \partial_j \left( D^2 \Phi \right) \right] = \Phi_{ik}^{\ell} \Phi_{j\ell}^k.
\label{eq_212} \end{equation}
Note that the positive semi-definite matrix $(g_{ij}(x))_{i,j=1,\ldots,n}$ is not necessarily invertible,
and it could happen that distinct points of $\RR^n$ have zero Riemannian distance with respect to the Riemannian metric $(g_{ij})$.
The metric $g_{ij}$ resembles an expression appearing in Lemma \ref{lem_1035}, a fact that will be exploited in the next section.

\section{Dualizing the Bochner inequality}
\label{sec_dual}

It is by now well-known that in the presence of convexity assumptions,
Poincar\'e-type inequalities may be deduced from Bochner's formula via a dualization procedure.
In this section we investigate the Poincar\'e inequality that is dual to Lemma \ref{lem_1035}.
This Poincar\'e inequality was also obtained in \cite{K_moment}, but in a cumbersome formulation and under an undesired assumption
called ``regularity at infinite'', which we eliminate here.

\medskip We begin with an easy case. Throughout this section we assume, in addition to the smoothness assumptions
made at the beginning of Section \ref{sec_gamma_2}, that there exists $\eps_0 > 0$ for which
\begin{equation}  D^2 \Phi(x) \geq \eps_0 \cdot {\rm Id} \qquad \qquad \qquad (x \in \RR^n) \label{eq_952} \end{equation}
in the sense of symmetric matrices.
Write $C_c^{\infty}(\RR^n)$ for the space of all compactly-supported, smooth functions on $\RR^n$.
The following lemma  is a variant of a well-known fact (see, e.g., Strichartz \cite{S}), that
compactly-supported functions are dense in Sobolev spaces when the Riemannian manifold is complete. Our assumption (\ref{eq_952})
implies the completeness of the Riemannian manifold $M = M_{\mu, \nu}$.

\begin{lemma} Let $f \in L^2(\mu)$ satisfy $\int f d \mu = 0$. Then there exists a sequence $u_k \in C_c^{\infty}(\RR^n)$ with
$$ \| Lu_k - f \|_{L^2(\mu)} \stackrel{k \rightarrow \infty}{\longrightarrow} 0. $$ \label{lem_1043}
\end{lemma}

\begin{proof} Recall that $\int (Lu) d \mu = 0$ for all $u \in C_c^{\infty}(\RR^n)$.
In order to show that the linear space $\{ L u \, ; \, u \in C_c^{\infty}(\RR^n) \}$ is dense,
we analyze its orthogonal complement. Let $f \in L^2(\mu)$ be in the orthogonal complement, i.e., for any $u \in C_c^{\infty}(\RR^n)$,
\begin{equation}  \int_{\RR^n} f (Lu) d \mu = 0. \label{eq_955}
\end{equation}
Our goal is to show that $f \equiv Const$. Note that (\ref{eq_955}) means that $f$ is a weak solution of $L f \equiv 0$.
Since $L$ is elliptic, then $f$ is smooth and $L f \equiv 0$ in the classical sense. Thus,
$$ L (f^2) = 2 f L f + 2 |\nabla_{M} f|^2  = 2 |\nabla_{M} f|^2. $$
Therefore, for any $\eta \in C_c^{\infty}(\RR^n)$,
\begin{align*}
 \int_{\RR^n} |\nabla_M( \eta f) |^2  d \mu & =
 \int_{\RR^n} \left[ \eta^2 |\nabla_M f |^2  + \frac{1}{2} \nabla_M (f^2) \cdot \nabla_M (\eta^2)  + f^2 |\nabla_M \eta|^2  \right] d \mu \\ & =
  \int_{\RR^n} \left[ \eta^2 |\nabla_M f |^2  - \frac{1}{2} \eta^2 L(f^2)   + f^2 |\nabla_M \eta|^2  \right] d \mu =   \int_{\RR^n} |\nabla_M \eta|^2 f^2   d \mu.
 \end{align*}
However, according to our assumption (\ref{eq_952}), we have $ |\nabla_M \eta|^2 = \Phi^{ij} \eta_i \eta_j \leq \eps_0^{-1} |\nabla \eta|^2. $
Let $\eta_R$ be a smooth cutoff function in $\RR^n$ that equals one on a Euclidean ball of radius $R$ centered at the origin, equals zero outside a Euclidean ball of radius $2R$,
and satisfies $|\nabla \eta_R| \leq 2/R$ throughout $\RR^n$. Then,
$$  \int_{K} |\nabla_M( \eta_R f) |^2 d \mu \leq \int_{\RR^n} |\nabla_M \eta|^2 f^2   d \mu \leq \eps_0^{-1} \int_{\RR^n}  |\nabla \eta_R|^2 f^2 d \mu \leq
\frac{2}{R \eps_0} \int_{\RR^n} f^2 d \mu \stackrel{R \rightarrow \infty}{\longrightarrow} 0, $$
since $f \in L^2(\mu)$. Therefore $\nabla f \equiv 0$ and $f$ is constant.
\end{proof}

Suppose that $F$ is a locally-Lipschitz function on a Riemannian manifold such as
$M_n^+(\RR)$.
By the Rademacher theorem, the gradient $\nabla F$ is well-defined almost everywhere
with respect to the Riemannian volume measure. In order to have a function $|\nabla F|$ that is defined everywhere,
in this note we set
\begin{equation}  |\nabla F|(x) = \limsup_{y\rightarrow x \atop{z\rightarrow x} } \frac{|f(y) - f(z)|}{dist(y,z)} = \lim_{\eps \rightarrow 0^+} \sup_{y,z \in B(x, \eps) \atop{y \neq z} }\frac{|f(y) - f(z)|}{dist(y,z)}
\label{eq_232} \end{equation}
where $dist$ is the Riemannian distance, and $B(x,\eps) = \{ y \, ; \, dist(x,y) < \eps \}$. Since $F$ is locally-Lipschitz, then the function $|\nabla F|$ is locally-bounded and upper semi-continuous.
Clearly, at any point $x$ where $F$ is continuously differentiable, $|\nabla F|(x)$ equals
the Riemannian length of $\nabla F(x)$.

\begin{proposition} Denote by $\theta$ the push-forward of the measure $\mu$ under the map (\ref{eq_1318}).
Then for any locally-Lipschitz function $F: M_n^+(\RR) \rightarrow \RR$ that belongs to $L^2(\theta)$ with $\int_{M_n^+(\RR)} F d \theta = 0$,
$$ \int_{M_n^+(\RR)} F^2 d \theta  \, \leq \, 4 \int_{M_n^+(\RR)} |\nabla F|^2 d \theta,
$$
whenever the right-hand side is finite.
\label{prop_poincare}
\end{proposition}

\begin{proof}
Since $F$ is locally-Lipschitz in $L^2(\theta)$, then the function $f$ defined via
$$  f(x) = F\left( D^2 \Phi(x) \right)
\qquad \qquad \qquad (x \in \RR^n), $$
is locally-Lipschitz in $\RR^n$ and belongs to $L^2(\mu)$.
Abbreviate $H = |\nabla F|^2$ and $h(x) = H \left(  D^2 \Phi (x) \right) $.
From the definition (\ref{eq_232}) of $|\nabla F|$, for any $x \in \RR^n$ in which $f$ is differentiable,
\begin{equation}
h(x) \geq  \sup \left \{
\sum_{i=1}^n V^i f_i \, ; \, \sum_{i,j=1}^n  g_{ij} V^i V^j \leq 1, \, V^1,\ldots,V^n \in \RR \right \},
\label{eq_1108}
\end{equation}
where $f_i$ and $g_{ij}$ are evaluated at the point $x$. In the case where the matrix $(g_{ij}(x))_{i,j=1,\ldots,n}$ is invertible,
we may express the supremum in (\ref{eq_1108}) in terms of the inverse matrix, yet it is the formula (\ref{eq_1108}) which is valid
in the general case. Setting $U_i = \Phi_{ij} V^j$, we reformulate (\ref{eq_1108}) as
\begin{equation}  h(x) \geq \sup \left \{
 \Phi^{ij} U_j f_i \, ; \,   g_{ij} \Phi^{k i} \Phi^{\ell j} U_k U_\ell \leq 1, \, U_1,\ldots,U_n \in \RR \right \}.
\label{eq_1004} \end{equation}
The formula (\ref{eq_1004}) is valid for almost any $x \in \RR^n$, since $f$ is differentiable almost everywhere in $\RR^n$ by the Rademacher theorem.
We would like to show that for any $u \in C_c^{\infty}(\RR^n)$,
\begin{equation}
-\int_{\RR^n} f (Lu) d \mu \leq 2 \sqrt{\int_{\RR^n} h^2 d \mu} \cdot \sqrt{\int_{\RR^n} (L u)^2 d \mu}. \label{eq_351}
\end{equation}
To that end, we observe that since $u$ is compactly-supported,
$$ \int_{\RR^n} \Gamma_2(u) d \mu = \frac{1}{2} \int_{\RR^n} L \left( \Phi^{ij} u_i u_j \right) d \mu - \int_{\RR^n} \Phi^{ij} (L u)_i u_j d \mu
= -\int_{\RR^n} \Phi^{ij} (L u)_i u_j d \mu = \int_{\RR^n} (L u)^2 d \mu. $$
Therefore Lemma \ref{lem_1035} and (\ref{eq_212}) imply that for any $u \in C_c^{\infty}(\RR^n)$,
$$ \int_{\RR^n} (L u)^2 d \mu \geq \frac{1}{4} \int_{\RR^n} \Phi^{i k} \Phi^{j \ell} g_{k \ell} u_i u_{j} d \mu. $$
Since $f$ is locally-Lipschitz,  we may safely integrate by parts and obtain that for any $u \in C_c^{\infty}(\RR^n)$,
\begin{align*}
 -\int_{\RR^n} & f (Lu) d \mu  = \int_{\RR^n} \Phi^{ij} f_i u_j d \mu \leq \int_{\RR^n} h(x) \sqrt{ g_{ij} \Phi^{k i} \Phi^{\ell j} u_k u_\ell  } d \mu(x) \\
 & \leq \sqrt{ \int_{\RR^n} h^2 d \mu } \sqrt{ \int_{\RR^n}  g_{ij} \Phi^{k i} \Phi^{\ell j} u_k u_\ell  \, d \mu  }
 \leq 2 \sqrt{ \int_{\RR^n} h^2 d \mu } \sqrt{ \int_{\RR^n} (Lu)^2 d \mu}
\end{align*}
and (\ref{eq_351}) is proven. Since $\int_{M_n^+(\RR)} F d \theta = 0$ then also $\int_{\RR^n} f d \mu = 0$. From Lemma
\ref{lem_1043} there exists a sequence $u_k \in C_c^{\infty}(\RR^n)$ with $L u_k \rightarrow -f$ in $L^2(\mu)$.
We substitute $u = u_k$ in  (\ref{eq_351}), and take the limit $k \rightarrow \infty$. This yields
$$ \int_{\RR^n} f^2 d \mu \leq 2 \sqrt{\int_{\RR^n} h^2 d \mu} \cdot \sqrt{\int_{\RR^n} f^2 d \mu}. $$
Hence,
$$ \int_{\RR^n} f^2 d \mu \leq 4 \int_K h^2 d \mu. $$
Since $h(x) = H(D^2 \Phi)$ with $H = |\nabla F|^2$,  the proposition is proven.
\end{proof}

\section{Regularity issues}
\label{sec_regularity}

This section explains how to eliminate  assumption (\ref{eq_952}) and also the smoothness assumptions of
the previous two sections.

\begin{theorem} Assume that $\mu$ and $\nu$
are  absolutely-continuous, log-concave probability measures on $\RR^n$.
Let $\nabla \Phi$
be the Brenier map between $\mu$ and $\nu$, and assume condition ($\star$)
from Section \ref{sec_intro}. Denote by $\theta$ the push-forward of the measure $\mu$ under
the map $x \mapsto D^2 \Phi(x)$.

\medskip
Then for any $\theta$-integrable, locally-Lipschitz function $F: M_n^+(\RR) \rightarrow \RR$,
\begin{equation} \int_{M_n^+(\RR)} F^2 d \theta  \, -  \, \left( \int_{M_n^+(\RR)} F d \theta \right)^2 \, \leq \, 4 \int_{M_n^+(\RR)} |\nabla F|^2 d \theta,
\label{eq_606}
\end{equation}
whenever the right-hand side is finite, and $|\nabla F|$ is interpreted as in (\ref{eq_232}).  \label{thm_539}
\end{theorem}

The strategy for proving Theorem \ref{thm_539} is to approximate
$\Phi$ by a sequence of functions $\Phi_N$ that satisfy assumption (\ref{eq_952}),
and to prove the pointwise (even local uniform) convergence
$D^2 \Phi_{N}(x) \stackrel{N \rightarrow \infty} \longrightarrow D^2 \Phi(x)$.
Below we discuss  two possible justifications of this convergence, as we believe that both of them may be useful.
The first proof occupies Subsection \ref{general_case}, and is based on various results from the regularity theory
of the Monge-Amp\`ere equation. The log-concavity of the measures is not really required for the first proof, and it suffices to assume that the densities are locally
H\"older.

\medskip The second proof in Subsection \ref{lc_target} is in fact
an alternative approach to  Caffareli's $C^{1,\alpha}$-regularity results in the log-concave case.
The argument in Subsection \ref{lc_target} is more self-contained, and it is based on integration-by-parts arguments.
The log-concavity of the target measure plays an important role here, and we further assume
a certain integrability condition on the logarithmic derivative of the density of $\mu$.
 This integrability condition is rather mild in our opinion, and it is satisfied
in many cases of interest.

\subsection{First proof of Theorem \ref{thm_539}}
\label{general_case}

As before, we write $e^{-V}$ and $e^{-W}$ for the densities of $\mu$ and $\nu$, respectively.
By log-concavity, the functions $V$ and $W$ are locally-Lipschitz in the open sets $Supp(\mu)$ and $Supp(\nu)$, respectively.
From condition ($\star$) the function $\Phi$ is $C^2$-smooth, and
the push-forward equation (\ref{eq_905}) implies that
\begin{equation}  \det D^2 \Phi(x) = e^{-V(x) + W(\nabla \Phi(x))}
\label{eq_2232} \end{equation}
for any $x \in Supp(\mu)$. In particular, $D^2 \Phi(x)$ is invertible, and hence positive-definite
for all $x \in Supp(\mu)$. Thus $\Phi$ is strictly-convex. The modulus of convexity of $\Phi$ at the point $x$ is defined to be
$$ \omega_{\Phi}(x; \delta) = \inf \left \{  \Phi(y) - \left( \Phi(x) + \nabla \Phi(x) \cdot (y - x) \right) \, ; \, y \in \RR^n, \,  |y-x| = \delta \right \}. $$
Then $\omega_{\Phi}(x; \delta)$ is a positive, continuous function of $x \in Supp(\mu)$ and $\delta > 0$, when we restrict attention
to $x$ and $\delta$ for which $\overline{B(x, \delta)} \subseteq Supp(\mu)$. Here, $B(x, \delta) = \{ y \in \RR^n \, ; \, |y-x| < \delta \}$.
Next, the Legendre transform
$$ \Phi^*(x) = \sup_{y \in \RR^n \atop{\Phi(y) < \infty}} \left[ x \cdot y \, - \, \Phi(y) \right] $$
is also $C^2$-smooth and strictly-convex in $Supp(\nu)$, with $y \mapsto \nabla \Phi^*(y)$
being the inverse map to $x \mapsto \nabla \Phi(x)$. Thus $\nabla \Phi$ is a $C^1$-diffeomorphism of $Supp(\mu)$
and $Supp(\nu)$. The reader is referred to Rockafellar \cite{Rockafellar}
for the basic properties of the Legendre transform.

\medskip We will approximate
$\mu$ and $\nu$ by sequences of probability measures $\mu_N$ and $\nu_N$ with the following properties:
\begin{enumerate}
\item[(i)] The probability measure $\mu_N$ (respectively $\nu_N$) has a density in $\RR^n$ of the form $e^{-V_N}$ (respectively $e^{-W_N}$).
\item[(ii)] The functions $V_N, W_N: \RR^n \rightarrow \RR$ are smooth and for any $x \in \RR^n$,
$$ D^2 V_N(x) \geq \frac{1}{N} \cdot {\rm Id}, \qquad D^2 W_N(x) \leq N \cdot {\rm Id}. $$
\item[(iii)] $V_N \longrightarrow V$ locally uniformly in $Supp(\mu)$, and similarly, $W_N \longrightarrow W$ locally uniformly in $Supp(\nu)$.
\end{enumerate}
It is quite standard to approximate $\mu$ and $\nu$ in this manner. For instance, in
order to obtain $\mu_N$ (or $\nu_N$), we may convolve $\mu$ (or $\nu$) with a Gaussian of a tiny variance, then multiply the resulting
density by a Gaussian of a huge variance, and then normalize to obtain a probability density.
Denote by $\nabla \Phi_N$ the Brenier map between $\mu_N$ and $\nu_N$. We use again Caffarelli's regularity theory,
to conclude that $\Phi_N: \RR^n \rightarrow \RR$ is a smooth, strictly-convex function, with
\begin{equation} \det D^2 \Phi_N(x) = e^{-V_N(x) + W_N(\nabla \Phi_N(x))} \qquad \qquad (x \in \RR^n). \label{eq_1058} \end{equation}
 The following lemma should be known to experts on the Monge-Amp\`ere equation,
yet we could not find it in the literature.

\begin{lemma} There exists an increasing sequence $\{ N_j \}$ such that
$$ D^2 \Phi_{N_j}(x) \stackrel{j \rightarrow \infty} \longrightarrow D^2 \Phi(x)
$$
locally uniformly in $x \in Supp(\mu)$. \label{lem_538}
\end{lemma}

\begin{proof} Fix $x_0 \in Supp(\mu)$. It suffices to find $\{ N_j \}$ such that
$D^2 \Phi_{N_j} \longrightarrow D^2 \Phi$ uniformly in a neighborhood of $x_0$.
A standard convexity argument (e.g., \cite[Section 2]{K_part_I}) based on (iii) and the fact that $\int e^{-V} = \int e^{-W} = 1$
shows that
there exist $A, B > 0$ with \begin{equation}
 \min \left \{ \inf_N V_N(x), \inf_N W_N(x), V(x), W(x) \right \} \geq A |x| - B,
 \qquad \qquad (x \in \RR^n).
\label{eq_1019} \end{equation}
Therefore,
\begin{equation} \sup_N \int_{\RR^n} |\nabla \Phi_N|^2 e^{-V_N(x)} dx = \sup_N \int_{\RR^n} |x|^2 e^{-W_N(x)} dx
\leq \int_{\RR^n} |x|^2 e^{B-A |x|} dx <
\infty.
\label{eq_1033} \end{equation}
Recall that $V_N \rightarrow V$ locally uniformly in $Supp(\mu)$, according to (iii).
From (\ref{eq_1033}) we learn that $\sup_N \| \Phi_N \|_{\dot{H}^1(K)} < \infty$
for any compact $K \subset Supp(\mu)$.
Here, $$ \| u \|_{\dot{H}^1(K)}^2 = \int_K |\nabla u(x)|^2 dx. $$
From the Rellich-Kondrachov compactness theorem (e.g., \cite[Section 4.6]{EG}),
we conclude that there exists a subsequence $\Phi_{N_j}$, numbers $C_j \in \RR$ and a certain function $F: Supp(\mu) \rightarrow \RR$ such that for any compact $K \subset Supp(\mu)$,
the sequence $\Phi_{N_j} + C_j$ converges to $F$ in $L^2(K)$. Passing to another subsequence, which
we conveniently denote again by $\{ \Phi_N \}$, and using \cite[Theorem 10.9]{Rockafellar},
we may assume that $F$ is convex and that the convergence is locally-uniform in $Supp(\mu)$.
Thus, from \cite[Theorem 24.5]{Rockafellar},
\begin{equation}  \nabla \Phi_N(x) \stackrel{N \rightarrow \infty}\longrightarrow \nabla F(x)
\label{eq_1046} \end{equation}
for almost any $x \in Supp(\mu)$. However, $(\nabla \Phi_N)_* \mu_N = \nu_N$. From
(iii), (\ref{eq_1019}) and (\ref{eq_1046}) we conclude that $(\nabla F)_* \mu = \nu$.
From the uniqueness of the Brenier map, we deduce that $\nabla F = \nabla \Phi$ almost everywhere in $Supp(\mu)$.
Since $\Phi$ is $C^2$-smooth, then we may apply  \cite[Theorem 25.7]{Rockafellar},
and upgrade (\ref{eq_1046}) to
\begin{equation}  \nabla \Phi_N(x) \stackrel{N \rightarrow \infty}\longrightarrow \nabla \Phi(x)
\label{eq_130} \end{equation}
locally uniformly in $Supp(\mu)$. The convexity arguments in \cite[Section 25]{Rockafellar}
also show that
$\nabla \Phi_N^* \rightarrow \nabla \Phi^*$ locally uniformly in $Supp(\nu)$.
As for the modulus of convexity, we have
\begin{equation} \omega_{\Phi_N}(x ; \delta) \stackrel{N \rightarrow \infty}{\longrightarrow}
\omega_{\Phi}(x ; \delta), \qquad \text{and respectively,} \qquad \omega_{\Phi^*_N}(y ; \delta) \stackrel{N \rightarrow \infty} \longrightarrow
\omega_{\Phi^*}(y ; \delta) \label{eq_1139} \end{equation}
locally uniformly in the set $\{ (x, \delta)  \in Supp(\mu) \times (0, \infty)  \, ; \, \overline{B(x,\delta)} \subset Supp(\mu) \}$,
and respectively, in the set $\{ (y, \delta)  \in Supp(\nu) \times (0, \infty)  \, ; \, \overline{B(y,\delta)} \subset Supp(\nu) \}$.

\medskip We will now invoke the estimates
of Gutierrez and Huang \cite{GH} and Forzani and Maldonado \cite{FM, FM2}, which are constructive
versions of Caffarelli's $C^{1,\alpha}$-regularity theory.
We are allowed to apply  \cite[Theorem 2.1]{GH}
and \cite[Theorem 15]{FM} locally near $x_0$, thanks to (iii), (\ref{eq_1058}), (\ref{eq_130}) and (\ref{eq_1139}).
From \cite[Theorem 15]{FM} we learn that there exist $\alpha, \delta, C > 0$ such that
for any $x, y \in B(x_0, \delta)$ and $N \geq 1$,
\begin{equation}  |\nabla \Phi_N(x) - \nabla \Phi_N(y)| \leq C |x - y|^{\alpha}.
\label{eq_326} \end{equation}
The function $V$ is locally-Lipschitz. From (iii) and \cite[Theorem 24.5]{Rockafellar},
the sequence $\{ V_N \}$ is uniformly locally-Lipschitz: This means that for any compact subset $K \subset Supp(\mu)$,
the Lipschitz constant of $V_N$ is bounded by some finite number $C_K$, independent of $N$.
Similarly, the sequence $\{ W_N \}$ is also uniformly locally-Lipschitz.
Together with (\ref{eq_130}) and (\ref{eq_326}) we deduce that there exist $\hat{C} > 0$ such that
$u_N(x) = -V_N(x) + W_N(\nabla \Phi_N(x))$ satisfies
$$ |u_N(x) - u_N(y)| \leq \hat{C} |x - y|^{\alpha}
\qquad \qquad (x, y \in B(x_0, \delta), N \geq 1).
$$
Recalling the Monge-Amp\`ere equation (\ref{eq_1058}), we learn that that there exists $\tilde{C} > 0$ such that
$$ \left| \det D^2 \Phi_N(x) - \det D^2 \Phi_N(y) \right| \leq \tilde{C} |x - y|^{\alpha}. \qquad \qquad (x, y \in B(x_0, \delta), N \geq 1).
$$
We are finally in good shape for applying the $C^{2,\alpha}$-estimates from Trudinger and Wang \cite[Theorem 3.2]{TW}.
These estimates yield the existence of $\bar{C} > 0$ such that
for any $x, y \in B(x_0, \delta/2)$ and $N \geq 1$,
\begin{equation}  \| D^2 \Phi_N(x) - D^2 \Phi_N(y) \|_{HS} \leq \bar{C} |x - y|^{\alpha}.
\label{eq_514} \end{equation}
The uniform $C^{2, \alpha}$-estimate in (\ref{eq_514}) allows us to apply the Arzella-Ascoli theorem. All we need is to
denote $K = B(x_0, \delta/2)$ and observe that
$$ \int_{K}  (\Delta \Phi_N) \xi = -\int_{ K} \nabla \Phi_N \cdot \nabla \xi \, \stackrel{N \rightarrow \infty}{\longrightarrow} \,
 -\int_{ K} \nabla \Phi \cdot \nabla \xi =  \int_{K}  (\Delta \Phi)  \xi, $$
 where $\xi$ is any smooth, compactly-supported function in $K$.
Hence the sequence $\{ \int_{K}  \Delta \Phi_N \}_{N \geq 1}$ is bounded, and since $D^2 \Phi_N$ is positive-definite, also
 the sequence $\{ \int_{K}  \| D^2 \Phi_N \|_{HS} \}_{N \geq 1}$ is bounded. From  (\ref{eq_514}) and the Arzella-Ascoli theorem, there exists a
 subsequence, denoted still by $\{ \Phi_N \}$, such that $D^2 \Phi_N \longrightarrow D^2 \Phi$ uniformly on $K = B(x_0, \delta/2)$.
\end{proof}

\begin{remark}{\rm Our proof of Lemma \ref{lem_538} does not make any use of the log-concavity of $\mu$ and $\nu$.
By inspecting the proof above, we see that Lemma \ref{lem_538} holds true as long as $V$ and $W$ are locally H\"older,
and $V_N, W_N$ are uniformly locally H\"older.
}\end{remark}

In order to simplify the notation, we denote the sequence $\{ \Phi_{N_j} \}$ from Lemma \ref{lem_538}
by $\{ \Phi_N \}$. Properties (i), (ii) and (iii) above are still satisfied.

\begin{corollary}  Denote by $\theta_N$ the push-forward of the measure $\mu_N$
under the map $x \mapsto D^2 \Phi_N(x)$. Then for any bounded, continuous function $b: M_n^+(\RR) \rightarrow \RR$,
\begin{equation}  \int_{M_n^+(\RR)} b d \theta_N \, \stackrel{N \rightarrow \infty}{\longrightarrow} \, \int_{M_n^+(\RR)} b d \theta.
\label{eq_202}
\end{equation}
Furthermore, if $b: M_n^+(\RR) \rightarrow \RR$ is bounded and upper semi-continuous, then
\begin{equation}  \limsup_{N \rightarrow \infty} \int_{M_n^+(\RR)} b d \theta_N \, \leq \, \int_{M_n^+(\RR)} b d \theta. \label{eq_205}
\end{equation}
\label{cor_539}
\end{corollary}

\begin{proof} In order to prove (\ref{eq_202}), we need to show that
$$ \int_{\RR^n} b \left( D^2 \Phi_N(x) \right) e^{-V_N(x)} dx \, \stackrel{N \rightarrow \infty}{\longrightarrow} \,
\int_{\RR^n} b \left( D^2 \Phi(x) \right) e^{-V(x)} dx. $$
This follows from Lemma \ref{lem_538} and the dominated convergence theorem, since (\ref{eq_1019}) provides an integrable majorant.
Next, assume that $b$ is bounded and upper semi-continuous. Then for any $x \in Supp(\mu)$,
$$ \limsup_{N \rightarrow \infty} b \left( D^2 \Phi_N(x) \right) e^{-V_N(x)} \,  \leq \, b(D^2 \Phi(x)) e^{-V(x)}. $$
Now (\ref{eq_205}) follows from Fatou's lemma, since we have an integrable majorant by (\ref{eq_1019}).
\end{proof}

\begin{proof}[Proof of Theorem \ref{thm_539}]
Assume first that
the locally-Lipschitz function $F$ is compactly supported.
We observe that for any fixed $N$,
assumption (\ref{eq_952}) holds true. Indeed, we may apply
a refinement of
Caffarelli's contraction theorem \cite{Caf3} which appears in \cite{kol1}, and obtain from (ii) that for any $x \in \RR^n$,
$$  D^2 \Phi_N(x) \geq \frac{1}{N^2} \cdot {\rm Id}.  $$
We may therefore apply Proposition \ref{prop_poincare}, and conclude that for any $N \geq 1$,
$$ \int_{M_n^+(\RR)} F^2 d \theta_N  \, -  \, \left( \int_{M_n^+(\RR)} F d \theta_N \right)^2 \, \leq \, 4 \int_{M_n^+(\RR)} |\nabla F|^2 d \theta_N.
$$
Recall that $|\nabla F|^2$ is upper semi-continuous and bounded, while $F$ is continuous and bounded. By taking the limit as $N \rightarrow \infty$ and using Corollary \ref{cor_539}, we obtain that
$$ \int_{M_n^+(\RR)} F^2 d \theta  \, -  \, \left( \int_{M_n^+(\RR)} F d \theta \right)^2 \, \leq \, 4 \int_{M_n^+(\RR)} |\nabla F|^2 d \theta,
$$
and (\ref{eq_606}) is proven in the case where $F$ is a compactly-supported function.

\medskip The next step is to prove  (\ref{eq_606}) under the additional assumption that $F \in L^2(\theta)$. To that end
we pick a smooth function $\theta_{R} : M_n^+(\RR) \rightarrow [0,1]$, such that $\theta_{R}$ equals one on $B({\rm Id}, R)$ and it vanishes outside $B({\rm Id}, 2R)$,
with $|\nabla \theta_R| \leq 2/R$. Set $F_R = \theta_R F$. We have just proven that (\ref{eq_606}) holds true when $F$ is replaced by $F_R$.
Clearly, $F_R \longrightarrow F$ in $L^2(\theta)$ as $R \longrightarrow \infty$. All that remains is to show that
\begin{equation} \limsup_{R \rightarrow \infty} \int_{M_n^+(\RR)} |\nabla F_R|^2 d \theta \leq \int_{M_n^+(\RR)} |\nabla F|^2 d \theta. \label{eq_245} \end{equation}
The functions $\theta_R$ and $F$ are continuous, and therefore we may use the Leibnitz rule $$ |\nabla F_R | \leq |F| |\nabla \theta_R| + \theta_R |\nabla F| \leq |\nabla F| + 2 |F| / R, $$
where we interpret $|\nabla F|$ and $|\nabla F_R|$ in the sense of definition (\ref{eq_232}). Since $F, |\nabla F| \in L^2(\theta)$, then (\ref{eq_245}) follows
in the case where $F \in L^2(\theta)$.

\medskip Finally, in order to eliminate the assumption that $F \in L^2(\theta)$, we replace
$F$ by $F_R = \max \{ -R, \min \{ F, R \} \}$, apply the inequality for $F_R$, and let
$R$ tend to infinity. For all but countably many values of $R$, the level set
 $\{ A \in M_n^+(\RR) \, ; \, F(A) = R \}$ has zero $\theta$-measure.
 Consequently, we have the inequality
 $\int |\nabla F_R|^2 d \theta \leq
\int |\nabla F|^2 d \theta$ for all but countably many values of $R$, and (\ref{eq_606}) follows.
\end{proof}

\subsection{Second proof: Log-concave target measure}
\label{lc_target}

In our second proof we will exploit the fact that $\nu$ is log-concave, but we will not require the log-concavity of $\mu$.
Throughout this subsection we make the following additional assumption:

\begin{enumerate}
\item[] {\bf Assumption (A):} For some $p > n$,
$$
\int_{\RR^n} |\nabla V|^{p} e^{-V} \ dx <\infty,
$$
where the derivatives $V_{i}$ are understood in the logarithmic derivative sense, i.e.
$$
\int_{\RR^n} \xi V_{i}  d \mu = - \int_{\RR^n} \xi_{i} d \mu, \quad \xi \in C^{\infty}_{c}(\mathbb{R}^n), i=1,\ldots,n.
$$
\end{enumerate}

By the Morrey embedding theorem (see, e.g., \cite[Section 4.5]{EG}), the function $V$
is locally  H{\"o}lder. We will approximate
$\mu$ and $\nu$ by sequences of probability measures $\mu_N$ and $\nu_N$ having properties (i), (ii) and
(iii) from Subsection \ref{general_case}. We also require a fourth property:
\begin{enumerate}
\item[(iv)] There exists $p>n$ such that
$$
\sup_{N} \int_{\RR^n} |\nabla V_N|^{p} e^{-V_N} \ dx <\infty.
$$
\end{enumerate}
The approach outlined in Subsection \ref{general_case}, to convolve with a tiny Gaussian and
then multiply by the density of a huge Gaussian, yields also property (iv).
Recall that the Brenier map $\nabla \Phi_N$ between $\mu_N$ and $\nu_N$ is  smooth and that it satisfies (\ref{eq_1058}).
The central ingredient of this subsection is the following a priori estimate:

\begin{proposition}
\label{glob-ae}
Assume that functions $V, W$ and $\Phi$ are smooth on the entire $\RR^n$ and that $\nu$ is a log-concave measure.
Then for every $q \ge 2$, $0 < \tau < 1, i=1,\ldots,n$ there exists $C(q,\tau)>0$
\begin{equation}
\label{Phi-vx}
\int_{\RR^n} \Phi_{ii}^{q} d \mu \le C(q,\tau) \Bigl( \int_{\RR^n} |V_i|^{\frac{2q}{2-\tau}} \ d \mu +
 \int_{\RR^n} |x_i|^{\frac{2q}{\tau}} \ d \nu \Bigr).
\end{equation}
\end{proposition}
\begin{proof}
Assume in addition that $D^2 W \ge \frac{1}{C} \cdot \rm{Id}$,  $D^2 V \le C \cdot \rm{Id}$. In this case $D^2 \Phi \le C^2 \cdot {\rm Id}$.
Recall formula (\ref{eq_445}),
$$
L(\Phi_{ii}) - \Phi^{j k}_i \Phi_{ijk} - \sum_{j, k=1}^n    \Phi_{j i} \Phi_{i k} W_{j k} \circ \nabla \Phi = -V_{ii},
$$
which is obtained by differentiating the change of variables formula (\ref{eq_2154}) along $x_i$.
Let us  multiply this formula by $\Phi^p_{ii}, p \ge 0$ and make a formal integration by parts  with respect to  $\mu$.  Using the convexity of $W$ we get
\begin{equation}
\label{28.01}
\int V_{ii} \Phi^{p}_{ii} \ d \mu
\ge   p \int \Phi^{p-1}_{ii} \langle (D^2 \Phi)^{-1}\nabla \Phi_{ii}, \nabla \Phi_{ii} \rangle d \mu
+ \int  \Phi^p_{ii} \Phi^{jk}_i \Phi_{ijk} d \mu.
\end{equation}
Let us justify this formula. To this end we fix a compactly supported function $\eta \ge 0$ and integrate with respect to $\eta \cdot \mu$.
$$
\int V_{ii} \Phi^{p}_{ii} \eta \ d \mu
\ge \int \langle (D^2 \Phi)^{-1} \nabla \eta, \nabla \Phi_{ii} \rangle \Phi^{p}_{ii} d \mu +   p \int \Phi^{p-1}_{ii} \langle (D^2 \Phi)^{-1}\nabla \Phi_{ii}, \nabla \Phi_{ii} \rangle  \eta  d \mu
+ \int  \Phi^p_{ii} \Phi^{jk}_i \Phi_{ijk}  \eta  d \mu.
$$
Applying the Cauchy inequality we get
$$
-\int \langle (D^2 \Phi)^{-1} \nabla \eta, \nabla \Phi_{ii} \rangle \Phi^{p}_{ii} d \mu
\le
 \frac{4}{\varepsilon}  \int \frac{\langle (D^2 \Phi)^{-1} \nabla \eta, \nabla \eta \rangle}{\eta} \Phi^{p+1}_{ii} d \mu + \varepsilon\int  \langle (D^2 \Phi)^{-1} \nabla \Phi_{ii}, \nabla \Phi_{ii} \rangle \Phi^{p-1}_{ii}  \eta  d \mu.
$$
Finally,
\begin{align*}
\int V_{ii} \Phi^{p}_{ii} \eta \ d \mu  & +  \frac{4}{\varepsilon}  \int \frac{\langle (D^2 \Phi)^{-1} \nabla \eta, \nabla \eta \rangle}{ \eta } \Phi^{p+1}_{ii} d \mu
\\& \ge   (p-\varepsilon) \int \Phi^{p-1}_{ii} \langle (D^2 \Phi)^{-1}\nabla \Phi_{ii}, \nabla \Phi_{ii} \rangle \eta d \mu
 + \int  \Phi^p_{ii} \Phi^{jk}_i \Phi_{ijk} \eta d \mu.
\end{align*}
 Assume that $\eta$ has the form
$\eta = \xi(\nabla \Phi)$, where $\xi$ is compactly supported. We get
$$\int V_{ii} \Phi^{p}_{ii} \eta \ d \mu  +   \frac{4 C^{p+2}}{\varepsilon}  \int \frac{|\nabla \xi|^2}{\xi} \ d \nu
\ge   (p-\varepsilon) \int \Phi^{p-1}_{ii} \langle (D^2 \Phi)^{-1}\nabla \Phi_{ii}, \nabla \Phi_{ii} \rangle \eta d \mu
+ \int  \Phi^p_{ii} \Phi^{jk}_i \Phi_{ijk} \eta d \mu.
$$
It remains to construct a sequence of functions $1 \ge \xi_N \ge 0$ satisfying $\lim_N \xi_N(x)=1$ for $\nu$-a.e. $x$ and $\lim_N \int |\nabla \xi_N|^2 / \xi_N \ d\nu=0$. Then applying the Fatou lemma we justify
(\ref{28.01}).

It is helpful to have in mind that $ \Phi^{jk}_i \Phi_{ijk} = \mbox{Tr} \Bigl[ ( D^2 \Phi)^{-\frac{1}{2}} D^2 \Phi_{i} ( D^2 \Phi)^{-\frac{1}{2}} \Bigr]^2 \geq 0$.
From (\ref{28.01}),
$$ \int V_{ii} \Phi^{p}_{ii} \ d \mu
\ge  p \int \Phi^{p-1}_{ii} \langle (D^2 \Phi)^{-1}\nabla \Phi_{ii}, \nabla \Phi_{ii} \rangle d \mu.
$$
Let us integrate by parts the left-hand side
$
\int V_{ii} \Phi^{p}_{ii} \ d \mu = \int V^2_{i} \Phi^{p}_{ii} \ d \mu
- p \int V_i \Phi^{p-1}_{ii} \Phi_{iii} \ d \mu.
$
The justification of this integration by parts is much easier, since $D^2 \Phi$ and $D^2 V$ are bounded.
Applying   $$ 2 |\Phi_{iii} V_i| \le 2 |V_i| \sqrt{\Phi_{ii}  \cdot \langle (D^2 \Phi)^{-1}\nabla \Phi_{ii}, \nabla \Phi_{ii} \rangle}
\le V^2_i \Phi_{ii} + \langle (D^2 \Phi)^{-1}\nabla \Phi_{ii}, \nabla \Phi_{ii} \rangle $$ one obtains

\begin{equation}
\label{09.01}
 \int V^2_{i} \Phi^{p}_{ii} \ d \mu  \ge  \int \Phi^{p-1}_{ii} \langle (D^2 \Phi)^{-1}\nabla \Phi_{ii}, \nabla \Phi_{ii} \rangle d \mu.
\end{equation}

Let us show that the right-hand side controls powers of the second derivative $\Phi_{ii}$.
Indeed, for every $q \ge 2$ and $\varepsilon >0,  0 \le \tau \le 1$ the following estimate holds
\begin{align*}
\int \Phi^q_{ii} \ d \mu & =  - (q-1) \int \Phi_{i}\Phi_{iii}  \Phi^{q-2}_{ii} \ d \mu + \int \Phi_i V_i \Phi^{q-1}_{ii} \ d \mu
\\&
 \le \varepsilon \int \Phi^2_i \Phi^{q-\tau}_{ii} \ d \mu + \frac{(q-1)^2}{4 \varepsilon} \int \Phi^{q-3 + \tau}_{ii} \langle (D^2 \Phi)^{-1}\nabla \Phi_{ii}, \nabla \Phi_{ii} \rangle d \mu
\\&
+ \frac{q-1}{q} \int  \Phi^{q}_{ii} \ d \mu + \frac{1}{q} \int |\Phi_i V_i|^q  \ d \mu.
\end{align*}
Finally,
\begin{align*}
 \int \Phi^q_{ii} \ d \mu
& \le
\int |\Phi_i V_i|^q  \ d \mu
+ q \varepsilon \int \Phi^2_i \Phi^{q-\tau}_{ii} \ d \mu +  \frac{q(q-1)^2}{4 \varepsilon} \int \Phi^{q-3 + \tau}_{ii} \langle (D^2 \Phi)^{-1}\nabla \Phi_{ii}, \nabla \Phi_{ii} \rangle d \mu
\
\\& \le
\int |\Phi_i V_i|^q  \ d \mu
+ q \varepsilon \int \Phi^2_i \Phi^{q-\tau}_{ii} \ d \mu +  \frac{q(q-1)^2}{4 \varepsilon} \int \Phi^{q-2 + \tau}_{ii}V^2_i d \mu.
\end{align*}
Applying H{\"o}lder inequalities $$\Phi^2_i \Phi^{q-\tau}_{ii} \le \frac{q-\tau}{q} \Phi^{q}_{ii} + \frac{\tau}{q} |\Phi_i|^{\frac{2q}{\tau}},$$
$$\Phi^{q-2 + \tau}_{ii}V^2_i \le \varepsilon \Phi_{ii}^q + C(\varepsilon,q, \tau) |V_i|^{\frac{2q}{2-\tau}},$$
$$ |\Phi_i V_i|^q \le \frac{2-\tau}{2} |V_i|^{\frac{2q}{2-\tau}} + \frac{\tau}{2} |\Phi_i|^{\frac{2q}{\tau}},$$
choosing sufficiently small $\varepsilon$, and applying the change of variables formula $\int |\Phi_i|^q d \mu = \int |x_i|^q d \nu$
we easily get the claim.

Finally, let us get rid of the assumption $D^2 W \ge \frac{1}{C} \cdot \rm{Id}$,  $D^2 V \le C \cdot \rm{Id}$.
To this end we approximate $\mu$ and $\nu$ by measures with smooth potentials satisfying
$D^2 W_N \ge \frac{1}{C_N} \cdot \rm{Id}$, $D^2 V_N \le C_N \cdot \rm{Id}$
satisfying $\lim_N \int |(V_N)_i|^{2q} \ d \mu_N =  \int |V_i|^{2q} \ d \mu$ and
$\lim_N \int |x_i|^{2q} \ d \nu_N = \int |x_i|^{2q} \ d \nu$.
It remains to show that the weak $L^q(\mu)$-limit of $(\Phi_N)_{ii}$ coincides with $\Phi_{ii}$.
The latter can be easily shown with the help of integration-by-parts and identifications of the poinwise limit $\lim_N \nabla \Phi_N$ with $\nabla \Phi$
(see the proof of Lemma \ref{lem_538}).
\end{proof}

\begin{remark}{\rm  The conclusion of Proposition \ref{glob-ae}
holds without any additional smoothness assumptions.
This can be verified by smooth approximations (see again  \cite{kol0}  for details).
Finally we get that (\ref{Phi-vx}) holds for every log-concave  measure $\nu$ and measure $\mu$ satisfying
$\int |V_i|^{\frac{2q}{2-\tau}} \ d \mu< \infty$, where $V_i$ is the logarithmic derivative of $\mu$ along $x_i$.}
\end{remark}

\begin{proof}[Second proof of Lemma \ref{lem_538}:] Let us show how Proposition \ref{glob-ae} implies
(\ref{eq_326}) above, without appealing to the works by Forzani and Maldonado \cite{FM, FM2} and
Gutierrez and Huang \cite{GH} related to Caffarelli's $C^{1,\alpha}$-regularity theory.
We use that  $\sup_N \int |\nabla V_N|^{p} e^{-V_N} \ dx < \infty$, $p>n$.  Since $\nu$ is log-concave, all the moments of $\nu$ are finite.
Thus Proposition \ref{glob-ae} implies $$\sup_N \int \|D^2 \Phi_N\|_{HS}^{p'} e^{-V_N} \ dx < \infty$$ for any $n < p'<p$.
Applying that $V_N$ are uniformly locally bounded from below, we get
 that $ \sup_N \int_{B_R} \|D^2 \Phi_N\|^{p'}_{HS} \ dx < \infty $ for every $R$. Then the result follows from the
Morrey embedding theorem.
\end{proof}

\section{Corollaries to Theorem \ref{thm_539}}
\label{sec6}

\begin{proof}[Proof of Theorem \ref{thm2}]  For $A \in M_n^+(\RR)$  define
$$ F(A) = f \left(\log \lambda_1(A), \ldots, \log \lambda_n(A) \right) $$
where $0 < \lambda_1(A) \leq \ldots \leq \lambda_n(A)$ are the eigenvalues of $A$.
According to Lemma \ref{lem_2200}, for any $A \in M_n^+(\RR)$,
\begin{equation}  |\nabla F|(A) \leq |\nabla f| \left(\log \lambda_1(A), \ldots, \log \lambda_n(A) \right). \label{eq_1013}
\end{equation}
Since $f$ is locally-Lipschitz and the eigenvalues vary continuously with the matrix $A$,
then (\ref{eq_1013}) implies that also $F$ is locally-Lipschitz.
Denote by $\theta$ the push-forward of the probability measure $\mu$
under the map $x \mapsto D^2 \Phi(x)$. Since $\EE \left| f(\Lambda(X)) \right| < \infty$
then $F \in L^1(\theta)$. Since $\EE |\nabla f|^2(\Lambda(X)) < \infty$,
then $\int |\nabla F|^2 d \theta < \infty$. We may apply Theorem \ref{thm_539} and conclude that
$$  \int_{M_n^+(\RR)} F^2 d \theta - \left( \int_{M_n^+(\RR)} F d \theta \right)^2
\leq 4 \int_{M_n^+(\RR)} |\nabla F|^2 d \theta.  $$
The left-hand side equals $Var \left[ f(\Lambda(X)) \right]$. Glancing at (\ref{eq_1013}), we thus obtain
$$ Var \left[ f(\Lambda(X)) \right] \leq 4 \EE |\nabla f|^2(\Lambda(X)), $$
and the proof is complete.
\end{proof}

\begin{proof}[Proof of Theorem \ref{thm1}] Plug in $f(x) = x_i$ in Theorem \ref{thm2}. Then $f$ is a $1$-Lipschitz
function, by Remark \ref{rem_1115}  we have $\EE \left| f(\Lambda(X)) \right| < \infty$. Thus the application of
Theorem \ref{thm2} is legitimate, and Theorem \ref{thm1} follows.
\end{proof}

\begin{proof}[Proof of Theorem \ref{thm3}] The argument is almost identical to the proof of Theorem \ref{thm1}, with Lemma \ref{lem_2201}
replacing the role of Lemma \ref{lem_2200}.
\end{proof}

Let us end this paper with a few remarks concerning future research.
If we make further assumptions regarding the log-concave measures in question,
it is possible to prove concentration inequalities for the eigenvalues
of $D^2 \Phi$ themselves, and not only for their logarithms. The analysis
of the weighted Riemannian manifold $M_{\mu, \nu}$ leads to such concentration
inequalities. Additionally, there is a soft argument which shows that
when $\nabla \Phi$ is the Brenier map between the uniform measure on $K$
and the uniform measure on $T$,
$$ \int_K \Delta \Phi \leq n V(K,\ldots,K, T), $$
where $V$ stands for mixed volume. The details will be discussed  elsewhere.
Another possible research direction is to investigate whether phenomena similar to Theorem
\ref{thm1} occur also in a non-linear setting, when transporting measures
with convexity properties supported on Riemannian manifolds.

{
}


\end{document}